\documentclass[12pt]{amsart}
\usepackage{amsmath,amsfonts,amssymb,amsthm}

\usepackage{hyperref}
\usepackage{graphicx}

\usepackage[margin=1.3in]{geometry}

\def\qed{\hfill\Box\smallskip}

\newtheorem{theorem}{Theorem}[section]
\newtheorem{lemma}[theorem]{Lemma}

\newtheorem{corollary}[theorem]{Corollary}
\newtheorem{proposition}[theorem]{Proposition}

\newtheorem{remark}[theorem]{Remark}

\newtheorem{question}{Question}

\def\({\left(}
\def\){\right)}

\DeclareMathOperator{\arccosh}{arccosh}

\DeclareMathOperator{\sign}{sign}
\DeclareMathOperator{\const}{const}
\DeclareMathOperator{\End}{End}

\begin{document}


\title[Asymptotics of tensor representations of the symmetric group]{Asymptotics of the maximal and the typical dimensions of isotypic components of tensor representations of the symmetric group}

\dedicatory{Dedicated to Professor Antonio Mach\`{i} on his 70th birthday}

\author{Sevak Mkrtchyan}

\email{sevak.mkrtchyan@rice.edu \newline\indent\url{http://math.rice.edu/~sm29}}
\address{Department of Mathematics MS 136, Rice University, Houston, TX 77005, USA }

\begin{abstract}

Vershik and Kerov gave asymptotical bounds for the maximal and the typical dimensions of irreducible representations of symmetric groups $S_n$. It was conjectured by G. Olshanski that the maximal and the typical dimensions of the isotypic components of tensor representations of the symmetric group admit similar asymptotical bounds. The main result of this article is the proof of this conjecture. Consider the natural representation of $S_n$ on $(\mathbb{C}^N)^{\otimes n}$. Its isotypic components are parametrized by Young diagrams with $n$ cells and at most $N$ rows. P. Biane found the limit shape of Young diagrams when $n\rightarrow\infty,\ \sqrt{n}/N\rightarrow c$. By showing that this limit shape is the unique solution to a variational problem, it is proven here, that after scaling, the maximal and the typical dimensions of isotypic components lie between positive constants. A new proof of Biane's limit-shape theorem is obtained.
\end{abstract}



\maketitle

\section{Introduction}

For $n\in\mathbb{N}$, let $S_n$ be the symmetric group on $n$ letters. The complex finite dimensional irreducible representations of $S_n$ are parametrized by the set of partitions of $n$, or equivalently by the set $\mathbb{Y}^n$ of Young diagrams with $n$ cells. Since each irreducible representation of $S_n$ appears in the left regular representation $\mathbb{C} S_n$ of $S_n$ with multiplicity equal to its dimension \cite{FultonHarris}, it follows that 
$$
n!=\sum_{\lambda\in\mathbb{Y}^n}\(\dim \lambda\)^2,
$$
where $\dim \lambda$ is the dimension of the irreducible representation $V_\lambda$ of $S_n$ corresponding to the Young diagram $\lambda$. Thus, the measure defined by 
$$
\mathbb{P}l^n(\lambda)=\frac{(\dim\lambda)^2}{n!}
$$
is a probability measure on $\mathbb{Y}^n$. $\mathbb{P}l^n$ is called the Plancherel measure. 

If $V$ is an irreducible subrepresentation of a representation $U$ of a finite group, the isotypic component of $U$ corresponding to $V$ is defined to be the sum of all subrepresentations of $U$ which are isomorphic to $V$. Isotypic components of $U$ are subrepresentations of $U$. $U$ decomposes uniquely into a direct sum of isotypic components. Note that following a widely used convention, whenever there is no ambiguity in the action, we will identify a representation with the underlying space. 

It is easy to see that for $\lambda\in\mathbb{Y}^n$, $\mathbb{P}l^n(\lambda)$ is the relative dimension of the isotypic component of the regular representation corresponding to $\lambda$. 

Two natural questions can be posed about the asymptotics of the dimensions of irreducible representations of the symmetric group:
\begin{question}
What is the asymptotic behavior of the maximal dimension of irreducible representations of $S_n$ in the limit $n\rightarrow\infty$?
\end{question}
\begin{question}
What is the asymptotic behavior of the dimension of a typical irreducible representation $V_\lambda$ of $S_n$ in the limit $n\rightarrow\infty$ if $\lambda$ is sampled randomly according to the Plancherel measure?
\end{question}

In 1985 Vershik and Kerov \cite{VK85} gave answers to both questions by obtaining two-sided logarithmically order-sharp asymptotic bounds. Vershik and Kerov conjectured that in the case of the typical dimension a stronger result holds: after appropriate scaling the dimensions of typical irreducible representations converge to a constant in measure. The conjecture has recently been proven by A. Bufetov \cite{Bu}.

\subsection{Main results}
The main results of this article are two-sided logarithmically order-sharp asymptotic bounds for the dimensions of isotypic components of tensor representations of $S_n$. Let $N,n$ be two positive integers and consider the tensor product space $(\mathbb{C}^N)^{\otimes n}$. The tensor representation of order $N$ of the symmetric group on $n$ letters is the natural action of $S_n$ on this space by permuting the factors in the tensor product. In coordinates, if $(v_1,v_2,\ldots,v_n)\in(\mathbb{C}^N)^{\otimes n}$ and $\pi\in S_n$, then
$$
\pi\cdot(v_1,v_2,\ldots,v_n)=(v_{\pi^{-1}(1)},v_{\pi^{-1}(2)},\ldots,v_{\pi^{-1}(n)}).
$$
It follows from Schur--Weyl duality (see Section \ref{sec:SchurWeyl}) that the irreducible representations which are subrepresentations of the representation $(\mathbb{C}^N)^{\otimes n}$ are exactly the ones which correspond to Young diagrams with $n$ cells and at most $N$ rows. Let $\mathbb{Y}_N^n$ denote the set of such Young diagrams, and given $\lambda\in\mathbb{Y}_N^n$ let $E_\lambda$ denote the isotypic component of $(\mathbb{C}^N)^{\otimes n}$ corresponding to $V_\lambda$. Looking at dimensions we have
\begin{equation}
N^n=\sum_{\lambda\in\mathbb{Y}_N^n}\dim E_\lambda.
\end{equation}
The relative dimensions of the isotypic components give a probability measure on $\mathbb{Y}_N^n$:
\begin{equation*}
\mathbb{P}_N^n(\lambda)=\frac{\dim(E_\lambda)}{N^n}.
\end{equation*}
The main results of this article are the following two theorems, conjectured by G. Olshanski, on the asymptotics of the dimensions of the isotypic components of tensor representations of the symmetric group in the limit $n,N\rightarrow\infty$, $\sqrt{n}/{N}\rightarrow c$.
\begin{theorem}
\label{thm:max}
For any $c>0$ there exist positive numbers $\alpha_c$ and $\beta$ such that for large enough $n\in\mathbb{N}$ and for any $N\in\mathbb{N}$, if $c>{\sqrt{n}}/{N}$, then
\begin{equation}
\alpha_c<-\frac{1}{\sqrt{n}}\ln\frac{\max_{\lambda\in\mathbb{Y}_N^n}\{\dim E_\lambda\}}{N^n}<\beta.
\end{equation}
\end{theorem}

\begin{theorem}
\label{thm:meas}
For any $c>0$ there exist positive numbers $\alpha_c$ and $\beta$ such that if $$\lim\limits_{n\rightarrow \infty}\frac {\sqrt{n}}{N}=c,$$ then
\begin{equation}
\lim_{n\rightarrow \infty}\mathbb{P}_N^n\left\{\lambda:\alpha_c<-\frac{1}{\sqrt{n}}\ln\frac{\dim E_\lambda}{N^n}<\beta\right\}=1.
\end{equation}
\end{theorem}

Note that the constants obtained in this article for both theorems are the same. In Section \ref{sec:MainProofs} we obtain exact formulas for the constants $\alpha_c$ and $\beta$. Also note that the bounds we obtain do not pretend to be sharp.

\subsection{Limit shape results}
\label{sec:LimitShapes}
The motivation behind considering the limit ${\sqrt{n}}/{N}\rightarrow c$ is a limit shape result by Vershik and Kerov \cite{VK77} and independently and simultaneously by Logan and Shepp \cite{LoganShepp} for random Young diagrams with respect to the Plancherel measure. The result has been generalized to the measures $\mathbb{P}_N^n$ by P. Biane \cite{Biane2001}. To state the results, we first need to introduce some notation.

Represent a Young diagram $\lambda$ with $n$ cells as a sequence $\lambda=(\lambda_1\geq\lambda_2\geq\ldots)$ where $\lambda_i\in\mathbb{Z}_{\geq 0}$ and $\sum\lambda_i=n$. Associate with $\lambda$ its diagram as shown in Figure \ref{fig:Diagram}. Here the longest row consists of $\lambda_1$ squares of size $1$, the next longest one of $\lambda_2$ squares, and so on. 

\begin{figure}[ht]
\centering
\includegraphics[width=6cm]{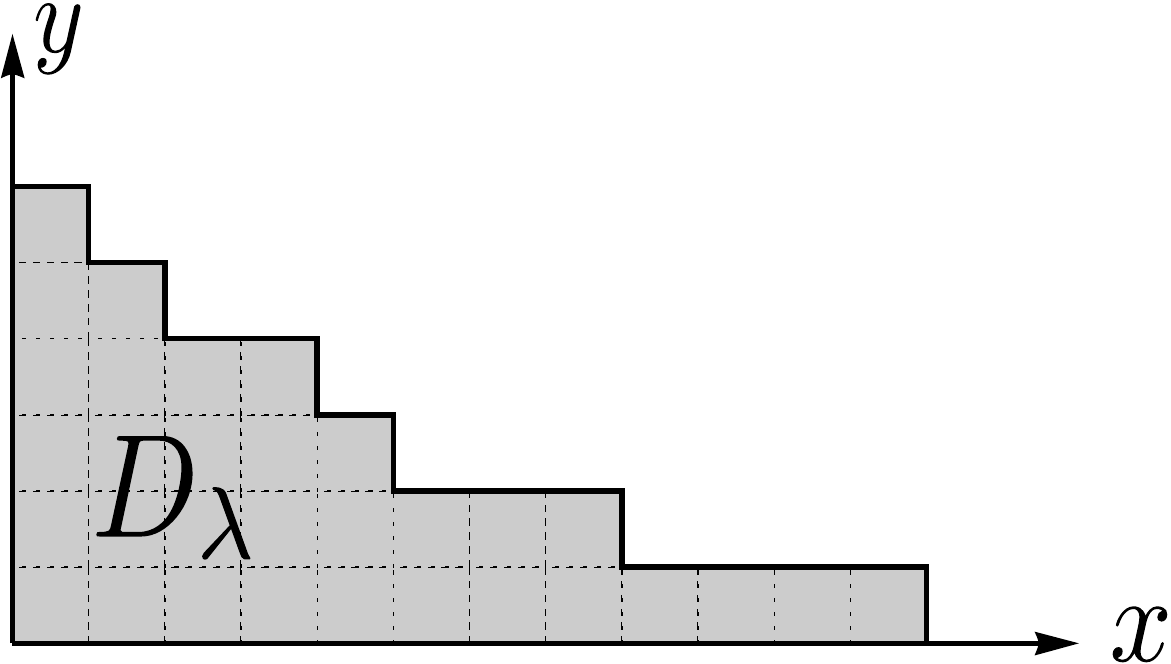}
\caption{\label{fig:Diagram} The Young diagram $\lambda=(12,8,5,4,2,1,0,0,\dots)$.}
\end{figure}
Scale the picture by $\sqrt{2n}$ in both directions so that the diagram has area $1/2$ and let $x$ and $y$ be the horizontal and vertical coordinates respectively. Rotate the scaled diagram by $\pi/4$ radians as in Figure \ref{fig:RotatedDiagram}. Let $X,Y$ be the horizontal and vertical coordinates in the rotated picture. We have $X=(x-y)/{\sqrt{2}}$ and $Y=(x+y)/{\sqrt{2}}$. Let $L_\lambda(X)$ be the function giving the top boundary of the rotated scaled diagram. $L_\lambda(X)$ is a piecewise linear function of slopes $\pm 1$ such that $L_\lambda(X)=|X|$ for $|X|\gg 1$. If $D_\lambda$ represents the interior of the scaled 
\begin{figure}[ht]
\centering
\includegraphics[width=7cm]{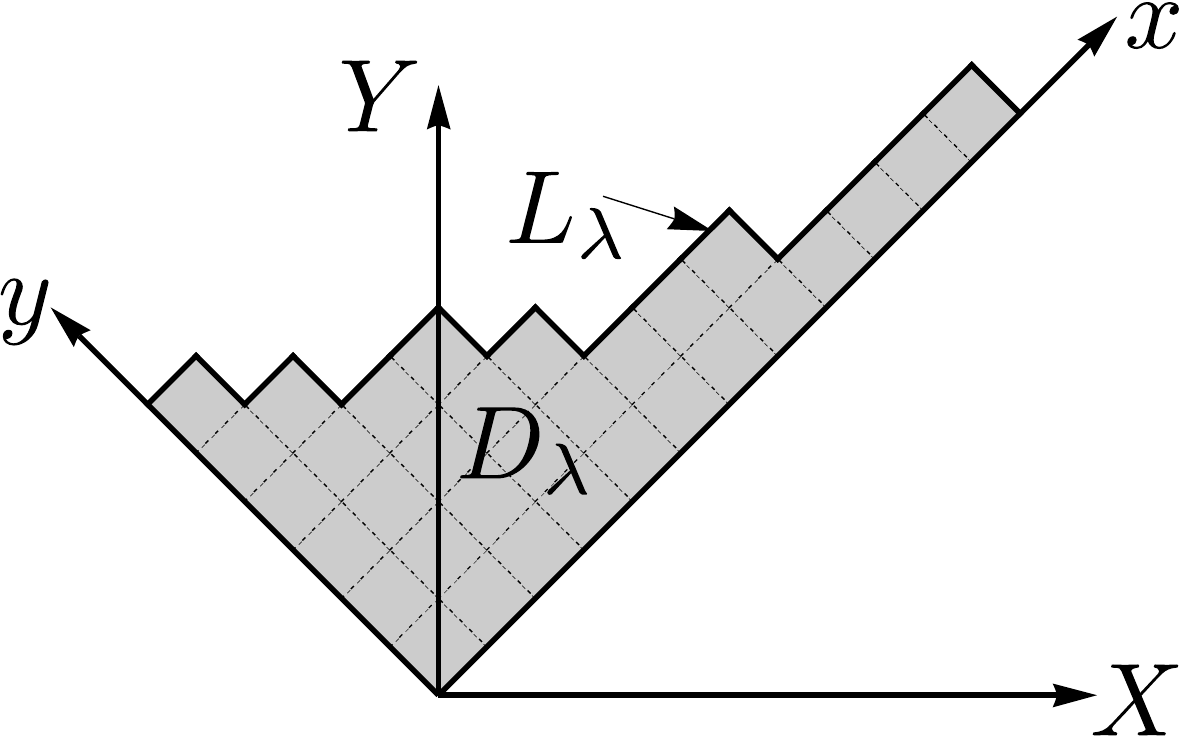}
\caption{\label{fig:RotatedDiagram} A rotated Young diagram.}
\end{figure}
Young diagram, then in the $(X,Y)$ coordinate system it can be characterized as
$$
D_\lambda=\{(X,Y):|X|\leq Y\leq L_\lambda(X)\}.
$$

In 1977 Vershik and Kerov and independently and simultaneously Logan and Shepp proved that the scaled random Young diagrams have a limit shape. 
\begin{theorem}[Vershik and Kerov \cite{VK77}, Logan and Shepp \cite{LoganShepp}]
\label{thm:VKLS}
For any $\varepsilon>0$
$$\lim_{n\rightarrow\infty} \mathbb{P}l^n\{\lambda\in\mathbb{Y}^n:|L_\lambda(X)-\Omega(X)|\leq\varepsilon, \forall X\in\mathbb{R}\}=1,$$
where $\Omega(X)$ is given by
\begin{equation*}
\Omega(X)=
\left\{\begin{array}{cl}
\frac 2\pi\left(\sqrt{1-X^2}+X\arcsin(X)\right),& |X|\leq 1,\\
|X|,&|X|>1.
\end{array}\right.
\end{equation*}
\end{theorem}
Note that $\Omega(X)$ has the rather simple derivative $$\Omega'(X)=\frac 2{\pi}\arcsin(X)\text{ for }|X|\leq1.$$

The measure $\mathbb{P}_N^n$ is a deformation of the Plancherel measure in the following way. When $N\geq n$, $(\mathbb{C}^N)^{\otimes n}$ contains a copy of the regular representation of $S_n$, whence a copy of each irreducible representation of $S_n$. As a consequence, when $N\geq n$, $\mathbb{Y}_N^n$ coincides with the set $\mathbb{Y}^n$ of all Young diagrams on $n$ cells. Moreover, in the limit $N\rightarrow \infty$ the measure $\mathbb{P}_N^n$ converges to the Plancherel measure on $\mathbb{Y}^n$ (see, for example, \cite[Section 3]{OlshNotes}). 

It follows from Theorem \ref{thm:VKLS} that the number of rows in a typical (with respect to the Plancherel measure) Young diagram with $n$ cells is of order $\sqrt{2n}$ \cite{VK85}. Thus, when studying asymptotic properties of Young diagrams with restricted number of rows and sampled according to the deformed Plancherel measures $P_N^n$, it is natural to consider the limit when the restriction on the number of rows grows on the order of the square root of the number of cells. 

P. Biane \cite{Biane2001} generalized the limit shape result for the Plancherel measure to the measures $\mathbb{P}_N^n$. Using methods of free probability theory he showed that in the limit $n,N\rightarrow\infty$, when ${\sqrt{n}}/{N}\rightarrow c\in[0,\infty)$, the shape of a typical scaled Young diagram chosen from $\mathbb{Y}_N^n$ according to the measure $\mathbb{P}_N^n$ converges to a curve $\Omega_c(s)$ which is a continuous deformation (depending on $c$) of the limit shape found in \cite{VK77} and \cite{LoganShepp}. When $c=0$, the limit shape $\Omega_c(s)$ is the Vershik-Kerov-Logan-Shepp limit shape: $\Omega_0(s)=\Omega(s)$. This is not surprising, since $c=0$ implies $N\gg\sqrt{2n}$, which means that the restriction on the number of rows in the Young diagrams $\lambda\in\mathbb{Y}_N^n$ is very weak. 

Before we state Biane's theorem exactly, let us define the function $\Omega_c(s)$. For $2s\in[c-2,c+2]$ define
\begin{multline*}
h(c,s):=\frac{2}{\pi} \(s \arcsin\(\frac{2 s + c}{2 \sqrt{1 + 2 s c}}\)
		\right.\\\left.
		+\frac{1}{2 c} \arccos\(\frac{2 + 2 s c - c^2}{2 \sqrt{1 + 2 s c}}\)+\frac 14 \sqrt{4 - (2 s - c)^2}\)
\end{multline*}
when $0<c<\infty$ and extend it continuously to $c=0$:
$$
h(0,s)=\frac{2}{\pi} \(s \arcsin(s)+\sqrt{1 - s^2}\).
$$
Define the function $\Omega_c(s)$ as follows:
\begin{figure}[t]
\centering
\includegraphics[width=12cm]{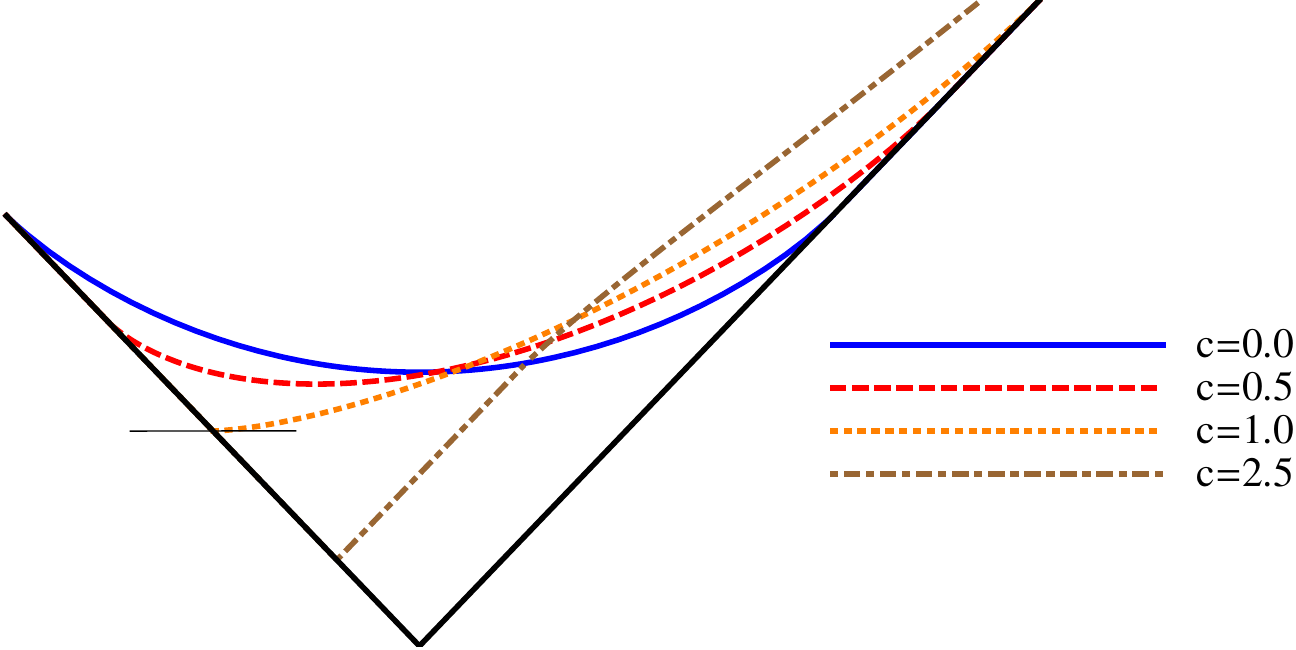}
\caption{\label{fig:Omegac}Graphs of $\Omega_c(s)$ for $c=0,0.5,1,2.5$.}
\end{figure}

\begin{equation*}
	\Omega_c(s)=\left\{
	\begin{array}{cl}
		h(c,s),&2s\in[c-2,c+2],
		\\|s|,&2s\notin[c-2,c+2]
	\end{array}\right.
\end{equation*}
if $0\leq c\leq 1$, and
\begin{equation*}
	\Omega_c(s)=\left\{
	\begin{array}{cl}
		s+\frac 1c ,&2s\in[-\frac 1c,c-2],
		\\h(c,s),&2s\in[c-2,c+2],
		\\|s|,&2s\notin[-\frac 1c,c+2],
	\end{array}\right.
\end{equation*}
if $c>1$. See Figure \ref{fig:Omegac} for graphs of the functions $\Omega_c(s)$ for several values of $c$. The graphs of the functions $\Omega_c(s)$ intersect the graph of $|s|$ at two points. All the intersections are tangential except the intersections on the left side for $c\geq 1$. At the left intersection point the graph of $\Omega_1(s)$ has slope $0$, while the graph of $\Omega_c(s)$ when $c>1$ has slope $1$.

Notice that $\Omega_c$ has a rather simple derivative: 
\begin{equation*}
\Omega_c'(s)=\frac 2\pi \arcsin\(\frac{c+2s}{2\sqrt{1+2cs}}\)
\end{equation*}
for $2s\in[c-2,c+2]$.

\begin{theorem}[P. Biane \cite{Biane2001}, Theorem 3]
\label{thm:Biane}
For any $\varepsilon>0$
$$\lim_{n,N\rightarrow\infty,\frac{\sqrt{n}}N\rightarrow c} \mathbb{P}_N^n\{\lambda\in\mathbb{Y}_N^n:|L_\lambda(X)-\Omega_c(X)|\leq\varepsilon, \forall X\in\mathbb{R}\}=1.$$
\end{theorem}

In this article we obtain a new proof of Biane's theorem.

\subsection{Outline of the article}
The first step is to obtain multiplicative formulas for the dimensions $\dim E_\lambda$. Schur--Weyl duality gives a characterization of $E_\lambda$ in terms of irreducible representations $V_\lambda$ and $W_\lambda$ of $S_n$ and the general linear group $GL(N,\mathbb{C})$ respectively, allowing us to express $\dim E_\lambda$ in terms of $\dim V_\lambda$ and $\dim W_\lambda$. For the dimensions of irreducible representations of $S_n$ we use the hook formula. For the dimensions of those irreducible representations of $GL(N,\mathbb{C})$ which appear in Schur--Weyl duality there are well-known multiplicative formulas (see Section \ref{sec:SchurWeyl}), which we use.

Taking the logarithm of $\dim E_\lambda$ the multiplicative formulas yield sums. The second step is to go from sums to integrals and calculate the correction terms, which we do in Section \ref{sec:IntegralFormula}. For the dimensions of irreducible representations of $S_n$ this was done by Vershik and Kerov \cite{VK85}.

The third and most difficult step is to prove that the integral part of $\dim E_\lambda$ has a unique minimizer and calculate the quadratic variation. The integral part can be viewed as a functional of the boundary function $L_\lambda$. In Section \ref{sec:UniqueMinimizer} we prove that the function $\Omega_c$ is the unique minimizer of this functional and prove that the quadratic variation is given by the $\frac 12$--Sobolev norm of $L_\lambda-\Omega_c$.

In Section \ref{sec:MainProofs} we present the proofs of the main theorems.

\subsection{Acknowledgements}
I am very grateful to Alexander Bufetov for many useful discussions on the subject. I am also very grateful to Grigori Olshanski for communicating this problem to me and for comments and suggestions.

\section{Schur--Weyl Duality}
\label{sec:SchurWeyl}

Notice that the general linear group $GL(N,\mathbb{C})$ also acts naturally on the tensor product space $(\mathbb{C}^N)^{\otimes n}$. In coordinates, if $(v_1,v_2,\ldots,v_n)\in(\mathbb{C}^N)^{\otimes n}$ and $A\in GL(N,\mathbb{C})$, then
$$
A\cdot(v_1,v_2,\ldots,v_n)=(Av_1,Av_2,\ldots,Av_n).
$$
It is easy to see that the actions of $S_n$ and $GL(N,\mathbb{C})$ commute. These actions give embeddings $S_n\hookrightarrow\End\left((\mathbb{C}^N)^{\otimes n}\right)$ and $GL(N,\mathbb{C})\hookrightarrow\End\left((\mathbb{C}^N)^{\otimes n}\right)$. Let $\mathfrak{a}_{S_n}$ and $\mathfrak{a}_{GL(N,\mathbb{C})}$ be the subalgebras of $\End\left((\mathbb{C}^N)^{\otimes n})\right)$ generated by the images of $S_n$ and $GL(N,\mathbb{C})$ respectively. Schur--Weyl duality \cite{W,FultonHarris} asserts that the subalgebras $\mathfrak{a}_{S_n}$ and $\mathfrak{a}_{GL(N,\mathbb{C})}$ are centralizers of each other in $\End\left((\mathbb{C}^N)^{\otimes n}\right)$. It follows \cite{FultonHarris} that the space $(\mathbb{C}^N)^{\otimes n}$ decomposes into a direct sum of tensor products of irreducible representations of the groups $S_n$ and $GL(N,\mathbb{C})$:
\begin{equation*}
(\mathbb{C}^N)^{\otimes n}=\bigoplus_{i\in I} V_i\otimes W_i,
\end{equation*}
where $V_i$-s are irreducible representations of $S_n$ and $W_i$-s are irreducible representations of $GL(N,\mathbb{C})$. Moreover, given $i\in I$, the isotypic component of $(\mathbb{C}^N)^{\otimes n}$ corresponding to $V_i$ is $V_i\otimes W_i$, and the same is true for $W_i$. It can also be obtained from Schur--Weyl duality \cite{FultonHarris} that the index set $I$ is the set $\mathbb{Y}_N^n$ of Young diagrams with $n$ cells and at most $N$ rows, and that given $\lambda\in I=\mathbb{Y}_N^n$, $W_\lambda$ is the irreducible highest weight representation of $GL(N,\mathbb{C})$ with highest weight $\lambda=(\lambda_1\geq\lambda_2\geq\ldots\geq\lambda_N)$.

As mentioned above, the isotypic components $E_\lambda$ are $E_\lambda = V_\lambda\otimes W_\lambda$. Thus, we have $\dim E_\lambda = \dim V_\lambda \cdot \dim W_\lambda$. The dimensions of the representations $V_\lambda$ are given by the hook formula. Given a Young diagram $\lambda$ and a pair of natural numbers $(i,j)$ we will say that $(i,j)\in\lambda$ if $j\leq \lambda_i$. For $(i,j)\in\lambda$ the cell $(i,j)$ is the cell in the $i$-th row and $j$-th column in the Young diagram $\lambda$. The hook of a cell $(i,j)$ is defined to be the set of cells to the right and above the cell, including the cell itself, as shown in Figure \ref{fig:hook}. 
\begin{figure}[ht]
\centering
\includegraphics[width=6cm]{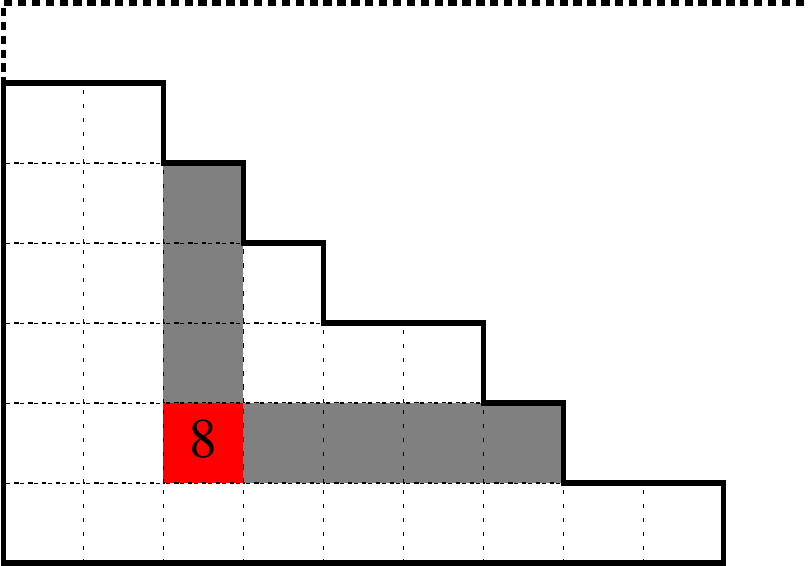}
\caption{\label{fig:hook} The hook and hook length of the cell $(2,3)$ in the Young diagram $\lambda=(9,7,6,4,3,2,0)\in\mathbb{Y}_7^{31}$.}
\end{figure}
The hook length $h_{i,j}$ of a cell $(i,j)\in\lambda$ is the number of cells in its hook. The following formula for $\dim V_\lambda$ is called the hook formula \cite{FultonHarris}:
\begin{equation}
\label{eq:HookFormula}
\dim V_\lambda=\frac{n!}{\prod_{(i,j)\in\lambda}h_{i,j}}.
\end{equation}

The content of the cell $(i,j)$ is defined to be $c_{i,j}:=j-i$. The dimension of the representation $W_\lambda$ is given by the following formula \cite{MacDonald}:
\begin{equation}
\label{eq:dimHighestWeight}
\dim W_\lambda = \frac{\prod_{(i,j)\in\lambda}(N+c_{i,j})}{\prod_{(i,j)\in\lambda}h_{i,j}}.
\end{equation}
Combining this with the hook formula we obtain
\begin{equation}
\label{eq:dimIso}
\dim E_\lambda = \frac{n!}{\prod_{(i,j)\in\lambda}h_{i,j}} \frac{\prod_{(i,j)\in\lambda}(N+c_{i,j})}{\prod_{(i,j)\in\lambda}h_{i,j}}.
\end{equation}

\section{An Integral Formula for the Measure}
\label{sec:IntegralFormula}
The goal of this section is to obtain an integral formula for $$\ln(\mathbb{P}_N^n(\lambda))=\ln\left(\frac{\dim E_\lambda}{N^n}\right).$$ For this purpose we need to introduce the continuous version of hook length. 

For a bounded region $d$ between the positive coordinate semi-axes and a top boundary given by a nonincreasing nonnegative function define the hook at $(x,y)\in d$ to be
$$
h_d(x,y):=\sup\{t:(x,t)\in d\}+\sup\{t:(t,y)\in d\}-x-y.
$$
\begin{figure}[ht]
\centering
\includegraphics[width=6cm]{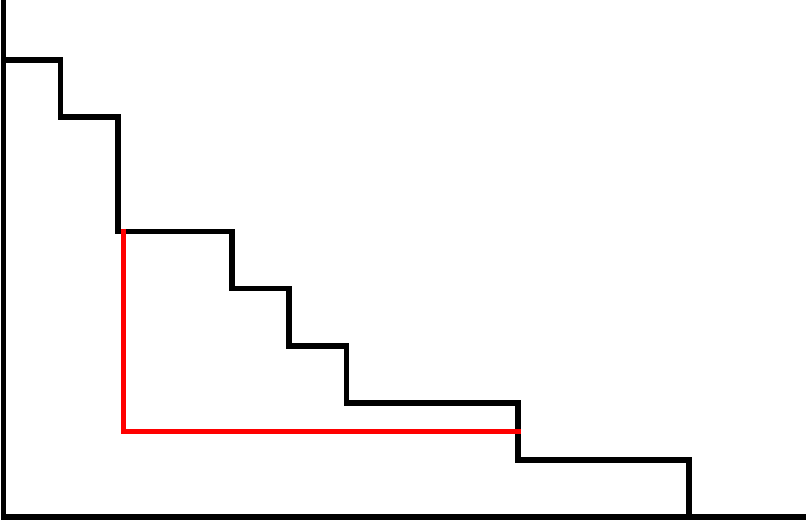}
\caption{\label{fig:ContHook} Continuous hook length.}
\end{figure}
For a Young diagram $\lambda$, $D_\lambda$ as defined in Section \ref{sec:LimitShapes} is such a bounded region, whence continuous hook length is defined for it (see Figure \ref{fig:ContHook}). To simplify the notation, we will denote $h_\lambda:=h_{D_\lambda}$. Introduce the coordinates $s$ and $t$ as $x=L_\lambda(t)+t$ and $y=L_\lambda(s)-s$. In these coordinates the hook at a point $(x,y)$ is given by $h_\lambda(x,y)=2(s-t)$ \cite{VK85}.

\begin{proposition}
\label{prop:measure}
For any $\lambda\in\mathbb{Y}_N^n$ we have
\begin{equation}
\label{eq:munN}
-\frac{\ln \mathbb{P}_N^n(\lambda)}{\sqrt{n}} = \sqrt{n}(\theta(\lambda)-\rho(\lambda))+\hat{\theta}(\lambda)-\hat{\rho}(\lambda)-\varepsilon_n,
\end{equation}
where
\begin{equation*}
\theta(\lambda)=1+2\iint_{(x,y)\in D_\lambda} \ln h_\lambda(x,y) dx dy,
\end{equation*}
\begin{equation*}
\rho(\lambda)=2\iint_{(x,y)\in D_\lambda}\ln\(1+\frac{\sqrt{2n}}{N}(x-y)\)dxdy,
\end{equation*}
$$\hat{\theta}(\lambda)=\frac {1}{\sqrt{n}}\sum_{(i,j)\in\lambda}m(h_{i,j}),$$
$$\hat{\rho}(\lambda)=\frac{1}{2\sqrt{n}}\sum_{(i,j)\in\lambda}m(N+c_{i,j}),$$
$$m(x)=\sum_{k=1}^\infty \frac{1}{k(k+1)(2k+1)}\frac 1{x^{2k}},$$
and $\varepsilon_n=o((\ln n)/{\sqrt{n}})$ is independent of $\lambda$.
\end{proposition}
\begin{remark}
$\theta(\lambda)$ is called the hook integral. Vershik and Kerov \cite{VK85} gave the following formula for $\theta(\lambda)$ in terms of $L_\lambda$:
\begin{equation}
\label{eq:ThetaInL}
\theta(\lambda)=\theta(L_\lambda):=1+2\iint_{t<s}\ln(2(s-t))(1-L_\lambda'(s))(1+L_\lambda'(t))dsdt.
\end{equation}
\end{remark}
\begin{remark}
The integrand in $\rho(\lambda)$ is constant along vertical lines in the rotated coordinate system, whence the double integral can be easily reduced to a single integral to give
\begin{equation}
\label{eq:RhoInL}
\rho(\lambda)=\rho(L_\lambda):=2\int_{-\infty}^{\infty}\ln\(1+\frac{2\sqrt{n}}{N}s\)(L_\lambda(s)-|s|)ds.
\end{equation}
\end{remark}
Note that originally $\theta$ and $\rho$ were defined as functions on Young diagrams. However, in light of \eqref{eq:ThetaInL} and \eqref{eq:RhoInL} we will treat them as functionals.

\begin{proof}[of Proposition \ref{prop:measure}]
Using \eqref{eq:HookFormula} the Plancherel measure $\mathbb{P}l^n(\lambda)=\frac{(\dim V_\lambda)^2}{n!}$ can be written as
\begin{equation}
\label{eq:plancherel}
\mathbb{P}l^n(\lambda)=\frac{n!}{\(\prod_{(i,j)\in\lambda}h_{i,j}\)^2},
\end{equation}
while using \eqref{eq:dimIso} the measure $\mathbb{P}_N^n(\lambda)=(\dim E_\lambda)/{N^n}$ can be written as
$$
\mathbb{P}_N^n(\lambda)= \frac{n!}{\(\prod_{(i,j)\in\lambda}h_{i,j}\)^2} \frac{\prod_{(i,j)\in\lambda}(N+c_{i,j})}{N^n}=\mathbb{P}l^n(\lambda)\prod_{(i,j)\in\lambda}\(1+\frac{c_{i,j}}N\).
$$
Thus,
$$
-\frac{\ln \mathbb{P}_N^n(\lambda)}{\sqrt{n}}=-\frac{\ln \mathbb{P}l^n(\lambda)}{\sqrt{n}}-\frac 1{\sqrt{n}}\sum_{(i,j)\in\lambda}\ln\(1+\frac{c_{i,j}}{N}\).
$$
Note that even though $c_{i,j}$ can be negative, since we are only considering Young diagrams with at most $N$ rows, we have that $1+{c_{i,j}}/N$ is positive.

It was shown in \cite{VK85} that 
$$
-\frac{\ln \mathbb{P}l^n(\lambda)}{\sqrt{n}}=\sqrt{n}\theta(\lambda)+\hat{\theta}(\lambda)-\varepsilon_n.
$$
Let $\square_{i,j}$ denote the $(i,j)$-th box in the scaled Young diagram, and let $(x_i,y_j)$ denote the center of $\square_{i,j}$. Note that the area of $\square_{i,j}$ is $1/{2n}$. Using this notation we obtain
\begin{align*}
-&\frac 1{\sqrt{n}}\sum_{(i,j)\in\lambda}\ln\(1+\frac{c_{i,j}}{N}\)+\sqrt{n}\rho(\lambda)
\\&=2\sqrt{n}\(\sum_{(i,j)\in\lambda}\iint_{\square_{i,j}}\ln\(1+\frac{\sqrt{2n}}{N}(x-y)\)dxdy-\frac 1{2n}\ln\(1+\frac{\sqrt{2n}}{N}(x_i-y_j)\)\)
\\&=2\sqrt{n}\sum_{(i,j)\in\lambda}\iint_{\square_{i,j}}\(\ln\(1+\frac{\sqrt{2n}}{N}(x-y)\)-\ln\(1+\frac{\sqrt{2n}}{N}(x_i-y_j)\)\)dxdy
\\&=2\sqrt{n}\sum_{(i,j)\in\lambda}\int_{x_i-\frac{1}{2\sqrt{2n}}}^{x_i+\frac{1}{2\sqrt{2n}}}\int_{y_j-\frac{1}{2\sqrt{2n}}}^{y_j+\frac{1}{2\sqrt{2n}}}
\\&\qquad\qquad\qquad\qquad
\ln\(1+\frac{\sqrt{2n}}{N\(1+\frac{\sqrt{2n}}{N}(x_i-y_j)\)}((x-x_i)-(y-y_j))\)dydx
\\&=2\sqrt{n}\sum_{(i,j)\in\lambda}\int_{-\frac{1}{2\sqrt{2n}}}^{\frac{1}{2\sqrt{2n}}}\int_{-\frac{1}{2\sqrt{2n}}}^{\frac{1}{2\sqrt{2n}}}\ln\(1+\frac{\sqrt{2n}}{N\(1+\frac{\sqrt{2n}}{N}(x_i-y_j)\)}(x-y)\)dydx.
\end{align*}

Denote $\alpha_{i,j}:=N(1+\frac{\sqrt{2n}}{N}(x_i-y_j))=N+c_{i,j}$. We have 
$$
-\frac 1{\sqrt{n}}\sum_{(i,j)\in\lambda}\ln\(1+\frac{c_{i,j}}{N}\)+\sqrt{n}\rho(\lambda)=
2\sqrt{n}\sum_{(i,j)\in\lambda}\frac{\alpha_{i,j}^2}{2n}\int_{-\frac{1}{2\alpha_{i,j}}}^{\frac{1}{2\alpha_{i,j}}}\int_{-\frac{1}{2\alpha_{i,j}}}^{\frac{1}{2\alpha_{i,j}}}\ln(1+x-y)dydx.
$$
From
$$
\iint\ln(1+x-y)dxdy=-\frac{(1+x-y)^2}{2}\ln(1+x-y)+\frac 34 (1+x-y)^2+C_1(x)+C_2(y)
$$
it follows that
\begin{multline*}
-\frac 1{\sqrt{n}}\sum_{(i,j)\in\lambda}\ln\(1+\frac{c_{i,j}}{N}\)+\sqrt{n}\rho(\lambda)
\\=\frac 1{2\sqrt{n}}\sum_{(i,j)\in\lambda} \(-3+(\alpha_{i,j}+1)^2\ln\(1+\frac1{\alpha_{i,j}}\)+(\alpha_{i,j}-1)^2\ln\(1-\frac 1{\alpha_{i,j}}\)\).
\end{multline*}

The power series expansion
$$
-3+\(1+\frac 1z\)^2\ln(1+z)+\(\frac 1z-1\)^2\ln(1-z)=-\sum_{k=1}^\infty\frac{1}{k(k+1)(2k+1)}z^{2k}
$$
completes the proof.
$\qed$
\end{proof}

\section{The Limit Shape is the Minimizer}
\label{sec:UniqueMinimizer}

Given a function $g$ let $\tilde{g}$ be the function defined by $\tilde{g}(x):=g(x+c/2)$. Throughout the text we will use the following shifted coordinates: 
\begin{equation}
\label{eq:shiftedNot}
z=s-\frac c2,\ \ w=t-\frac c2.
\end{equation}
In particular, for any function $g$ we have $g(s)=\tilde{g}(z)$.

We will use the following notation:
$$
\delta_{S}=\left\{\begin{array}{ll}1,&S\text{ is true}\\0,&S\text{ is false}\end{array}\right.
$$
and
$$
\sign(x)=\left\{\begin{array}{cl}-1,&x<0\\0,&x=0\\1,&x>0\end{array}\right. .
$$

\begin{proposition}
\label{prop:integral}
Let $c=c_n={\sqrt{n}}/{N}>0$. Let $L(X)$ be an arbitrary continuous and piecewise differentiable function satisfying the following conditions (see Figure \ref{fig:L}):
\begin{figure}[t]
\centering
\includegraphics[width=7cm]{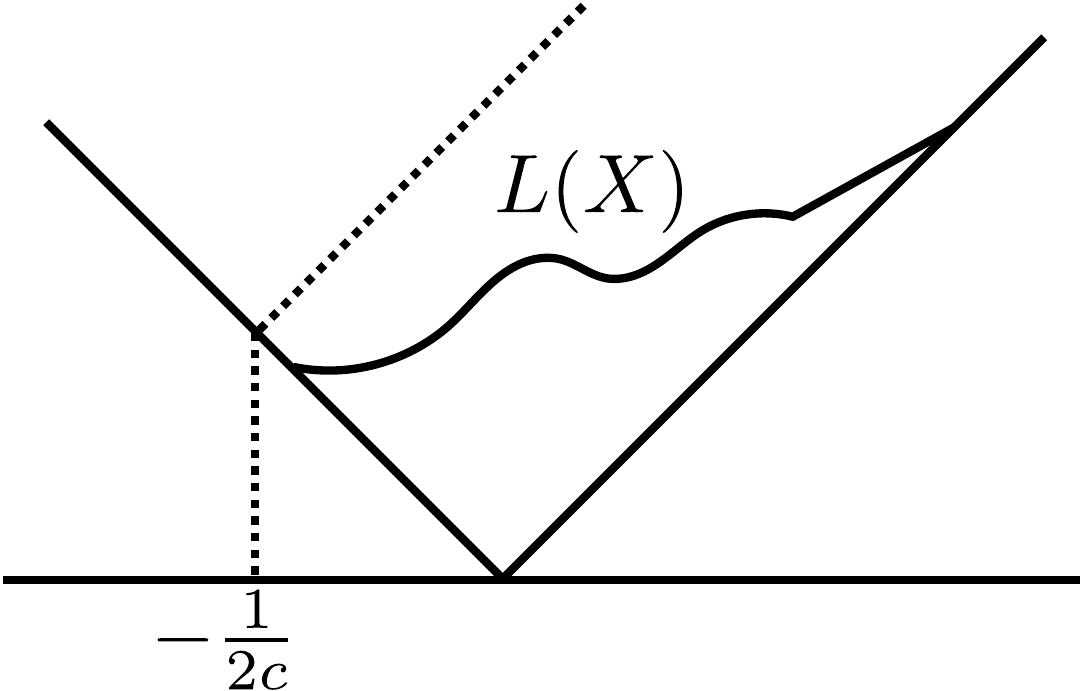}
\caption{The graph of $L(X)$.}
\label{fig:L}
\end{figure}
\begin{equation}
\label{eq:CondOnL}
\begin{array}{l}
 L(X)\geq |X|\text{ for all }X,
 \\|L'(X)|\leq 1\text{ for all }X\text{ where }L(X)\text{ is differentiable},
 \\L(X)=|X|\text{ for }X\gg 1\text{ and }X\leq-\frac{1}{2c},
 \\L(X)<X+\frac{1}{c}\text{ for }X\geq-\frac{1}{2c},
 \\\int_{-\infty}^\infty (L(X)-|X|)dX=\frac 12.
\end{array}
\end{equation}
Then
\begin{equation}
\label{eq:thetaRho-norm}
\theta(L)-\rho(L)=\frac 12 \|f\|_{\frac 12}^2+2\int_{|s-\frac c2|>1}H_c'(s)f(s)ds,
\end{equation}
where $f(s)=L(s)-\Omega_c(s)$,
\begin{multline*}
\tilde{H}_c(z):=
\delta_{|z|>1}\left(\left(z-\frac {1-c^2}{2c}\right) \arccosh|z|
\right.\\\left.\quad\quad\quad\quad +\sign(1-c)\left(z+\frac {1+c^2}{2c}\right)\arccosh\left|
      \frac{1+\frac{1+c^2}{2c}z}{z+\frac{1+c^2}{2c}}\right|-\sign(z)\sqrt{z^2-1}\right)
\end{multline*}
and
$$
\|f\|_{\frac 12}^2=\int_{-\infty}^{\infty}\int_{-\infty}^{\infty}{\(\frac{f(s)-f(t)}{s-t}\)^2}dsdt
$$
is the $\frac 12$--Sobolev norm in the space of piecewise-smooth functions.
\end{proposition}
\begin{proof}
Define
$$
	\phi_0(x):=-\ln|2x|,\ \phi_k(x):=\int_0^x{\phi_{k-1}(y)}dy.
$$
Choose numbers $a<min\{-1+c/2,-{1}/(2c)\}$ and $b>1+ c/2$ such that $f(s)=0$ when $s\notin(a,b)$. This is possible since $L(s)=\Omega_c(s)=|s|$ when $|s|\gg 1$. Using \eqref{eq:ThetaInL} write $\theta(L)$ as follows
\begin{align}
	\theta(L)
	=&1-\int_a^b\int_a^b{\phi_0(s-t)}dsdt
		-2\int_a^b{\phi_1(a-t)L'(t)}dt
	\\\nonumber
	 &-2\int_a^b{\phi_1(b-t)L'(t)}dt
		+\int_a^b\int_a^b{\phi_0(s-t)L'(s)L'(t)}dsdt
	\\\nonumber
	=&-\int_{-\infty}^{\infty}\int_{-\infty}^{\infty}{\ln|2(s-t)|f'(s)f'(t)}dsdt+\theta_1(L)+\theta_2(L),
\end{align}
where
$$
	\theta_1(L):=1-\int_a^b\int_a^b{\phi_0(s-t)}dsdt-\int_a^b\int_a^b{\phi_0(s-t)\Omega_c'(s)\Omega_c'(t)}dsdt,
$$
$$
	\theta_2(L):=2\int_a^b{\(I_c(s)-\phi_1(a-s)-\phi_1(b-s)\)L'(s)}ds
$$
and
$$
	I_c(s):=\int_a^b{\phi_0(s-t)\Omega_c'(t)}dt.
$$
Vershik and Kerov \cite[Lemma 4]{VK85} have shown that $$-\int_{-\infty}^{\infty}\int_{-\infty}^{\infty}{\ln|2(s-t)|f'(s)f'(t)}dsdt=\frac 12 \|f\|_{\frac 12}^2.$$

Lemma \ref{lem:intIOmega'} below implies that
\begin{multline*}
\theta_1(L)=1-2\phi_2(b-a)-\int_a^b{I_c(s)\Omega_c'(s)}ds=\frac{c^2}4-2\int_a^b{G_c(s)\Omega_c'(s)}ds\\+2\int_a^b{H_c(s)\Omega_c'(s)}ds
-\delta_{c>1}\(1-\frac 5{4c^2}+\frac{c^2}4-\(2+\frac 1{c^2}\)\ln(c)\)
\end{multline*}
while Lemma \ref{lem:I} implies that
$$
\theta_2(L)=2\int_a^b{G_c(s)L_c'(s)}ds-2\int_a^b{H_c(s)L_c'(s)}ds
$$
(see \eqref{eq:G} for the definition of $G$).

Integrating by parts we obtain
\begin{multline*}
\theta_1(L)+\theta_2(L)=\frac{c^2}4+2(G_c(s)-H_c(s))f(s)\big|_a^b-2\int_a^b{(G_c'(s)-H_c'(s))f(s)}ds\\-\delta_{c>1}\(1-\frac 5{4c^2}+\frac{c^2}4-\(2+\frac 1{c^2}\)\ln(c)\).
\end{multline*}

Since $f(a)=f(b)=0$ and $G'_c(s)=-\ln|1+2cs|$, the above expression simplifies to
\begin{multline*}
\theta_1(L)+\theta_2(L)=\frac{c^2}4+2\int_a^b\ln|1+2cs|f(s)ds+2\int_a^b{H_c'(s)f(s)}ds\\-\delta_{c>1}\(1-\frac 5{4c^2}+\frac{c^2}4-\(2+\frac 1{c^2}\)\ln(c)\).
\end{multline*}

Combining the above results with formula \eqref{eq:RhoInL} for $\rho(L)$ we obtain
\begin{multline*}
\theta(L)-\rho(L)=\frac 12 \|f\|_{\frac 12}^2+2\int_a^b{H_c'(s)f(s)}ds+\frac{c^2}4
+2\int_{\frac c2-1}^{\frac c2+1}{\ln(1+2cs)(|s|-\Omega_c(s))}ds
\\-\delta_{c>1}\(1-\frac 5{4c^2}+\frac{c^2}4-\(2+\frac 1{c^2}\)\ln(c)\).
\end{multline*}
It follows from Lemma \ref{lem:F2-2} that the sum of all the terms except the first two is zero. This completes the proof.
$\qed$
\end{proof}

\begin{corollary}
\label{cor:UniqueMinimizer}
If $L$ is any function satisfying the conditions \eqref{eq:CondOnL} then $\theta(L)-\rho(L)\geq0$ with equality holding if and only if $L=\Omega_c$.
\end{corollary}
\begin{proof}
It is immediate that the first term on the left-hand side of \eqref{eq:thetaRho-norm} is nonnegative. Thus, it is enough to show that the integrand in the second term is nonnegative as well. Recall that $\Omega_c(s)=s$ if $s>1+ c/2$. If $0<c<1$, then we also have $\Omega_c(s)=|s|$ whenever $s<-1+ c/2$. Since $L(s)\geq|s|$ for all $s$, we have
\begin{equation}
\label{eq:OmegaMinusLPositiveSmallc}
L(s)-\Omega_c(s)\geq 0\text{ if }0<c<1\text{ and }\left|s-\frac c2\right|>1. 
\end{equation}
If $c\geq 1$ and $s\in(-1/(2c),-1+c/2)$, we have that $\Omega_c(s)-L(s)<0$, since $\Omega_c(s)=s+1/c$ for such $s$ and $L(s)<s+1/c$ when $s\geq -1/(2c)$. Thus, we have
\begin{equation}
\label{eq:OmegaMinusLPositiveLargec}
\begin{array}{c}
L(s)-\Omega_c(s)\geq 0\text{ if }c\geq 1\text{ and }s>1+\frac c2,
\\\text{and}\\
L(s)-\Omega_c(s)\leq 0\text{ if }c\geq 1\text{ and }s\in(-\frac 1{2c},-1+\frac c2).
\end{array}
\end{equation}
Since $L(s)-\Omega_c(s)=0$ for $s<- 1/(2c)$, in order to show \eqref{eq:thetaRho-norm} is nonnegative it is enough to show that $H'_c(s)$ also satisfies the inequalities \eqref{eq:OmegaMinusLPositiveSmallc} and \eqref{eq:OmegaMinusLPositiveLargec}. In the shifted notation we need to show that
\begin{equation}
\begin{array}{c}
\label{eq:HpcSigns}
\tilde{H}'_c(z)\geq 0\text{ when }z\in(1,\infty)\text{ or } 0<c<1\text{ and }z\in\left(-\frac {1+c^2}{2c},-1\right),
\\\text{and}\\
\tilde{H}'_c(z)\leq 0\text{ when }c\geq 1\text{ and }z\in\left(-\frac {1+c^2}{2c},-1\right).
\end{array}
\end{equation}
Differentiating $\tilde{H}_c(z)$ when $|z|>1$ we obtain
\begin{equation}
\label{eq:dH}
\tilde{H}'_c(z)=\arccosh|z|+\sign(1-c)\arccosh\left|
      \frac{1+\frac{1+c^2}{2c}z}{z+\frac{1+c^2}{2c}}\right|,
\end{equation}
which implies that
\begin{equation}
\label{eq:limdH1}
\lim_{z\rightarrow -1^-\text{ or }z\rightarrow 1^+} \tilde{H}'_c(z)=0.
\end{equation}
Differentiating \eqref{eq:dH} we obtain
\begin{equation}
\label{eq:ddH}
\tilde{H}''_c(z)=\frac{\sign(z)\(z+\frac 1c\)}{\(\frac{1+c^2}{2c}+z\)\sqrt{z^2-1}}.
\end{equation}
When $0<c<1$, we have that $\frac 1c>\frac{1+c^2}{2c}$, whence \eqref{eq:ddH} implies that $\sign(\tilde{H}''_c(z))=\sign(z)$ for $z\in(-\frac {1+c^2}{2c},-1)\cup(1,\infty)$. When $c\geq 1$, we have 
$\tilde{H}''_c(z)>0$ for $z\in(-\frac {1+c^2}{2c},-1)\cup(1,\infty)$. These, together with \eqref{eq:limdH1} imply \eqref{eq:HpcSigns}.

That the equality $\theta(L)-\rho(L)=0$ holds if and only if $L=\Omega_c$ follows immediately from the above arguments.

$\qed$
\end{proof}

\begin{remark}
Biane's theorem (Theorem \ref{thm:Biane}) is an immediate corollary of Proposition \ref{prop:measure} and Corollary \ref{cor:UniqueMinimizer}.
\end{remark}

\subsection{Proofs of the lemmas}
\label{sec:proofsOfLemmas}
In this section we list the lemmas used in the proof of Proposition \ref{prop:integral} and give their proofs. Before we move on, let us give two integrals which will be used throughout the proofs (both formulas can be easily obtained from \cite[2.266, p.97]{GradshteynRyzhik}):
\begin{equation}
\label{eq:GRzSmall}
	\int_{-1}^1{\frac 1{(\alpha-z)\sqrt{1-z^2}}}dz=
	\left\{
	\begin{array}{cr}
		\sign(\alpha)\frac \pi{\sqrt{\alpha^2-1}}
		,&|\alpha|>1
		\\0
		,&|\alpha|\leq1
	\end{array}\right.=\delta_{|\alpha|>1}Sign(\alpha)\frac \pi{\sqrt{\alpha^2-1}}
\end{equation}
and if $\alpha>1$,
\begin{equation}
\label{eq:GRzBig}
  \int{\frac 1{(z+\alpha)\sqrt{z^2-1}}}dz=
  \frac{\sign(z)}{\sqrt{\alpha^2-1}}\arccosh\left|\frac{1+\alpha z}{z+\alpha}\right|+\const.
\end{equation}
In particular, setting $\alpha=-\frac{1+c^2}{2c}, c>0$ in \eqref{eq:GRzSmall} we obtain
\begin{equation}
\label{eq:GRzSmallc}
\int_{-1}^1{\frac 1{\(\frac{1+c^2}{2c}+z\)\sqrt{1-z^2}}}dz=\frac{2c\pi}{|1-c^2|}.
\end{equation}

\begin{lemma}
\label{lem:F2-2}
Let $A(c)$ be the following integral:
\begin{equation*}
A(c):=\int_{a}^{b}{\ln(1+2cs)(|s|-\Omega_c(s))}ds.
\end{equation*}
For any $c>0$ we have
\begin{equation}
\label{eq:F2-2}
A(c)=-\frac{c^2}8+\delta_{c>1}\(\frac 12-\frac 5{8c^2}+\frac{c^2}8-\(1+\frac 1{2c^2}\)\ln(c)\).
\end{equation}
\end{lemma}
\begin{proof}
Switching to the shifted notation \eqref{eq:shiftedNot} we obtain
\begin{equation}
\label{eq:lnOmega}
A(c)=\int_{a-\frac c2}^{b-\frac c2}{\ln(1+c^2+2cz)\(\left|z+\frac c2\right|-\tilde{\Omega}_c(z)\)}dz.
\end{equation}
Integrating \eqref{eq:lnOmega} by parts we obtain:
\begin{equation*}
A(c)=\int_{a-\frac c2}^{b-\frac c2}\frac 1c\phi_1\(\frac{1+c^2+2cz}2\)\(\sign\(z+\frac c2\)-\tilde{\Omega}'_c(z)\)dz.
\end{equation*}
Integrating by parts a second time, noting the discontinuity at $z=-c/2$ and noting that $\tilde{\Omega}_c''(z)=0$ for $|z|>1$, we obtain
\begin{equation*}
A(c)=-\frac 2{c^2}\phi_2\(\frac 12\)+\int_{-1}^{1}\frac 1{c^2}\phi_2\(\frac{1+c^2+2cz}2\)\tilde{\Omega}''_c(z)dz,
\end{equation*}
whence \eqref{eq:F2-2} is equivalent to
\begin{equation}
\label{eq:intphi2p''}
\int_{-1}^1{\phi_2\(\frac{1+c^2+2cz}2\)\tilde{\Omega}_c''(z)}dz=
\left\{\begin{array}{ll}
\frac 38-\frac{c^4}8,&0<c\leq 1\\
-\frac 14+\frac {c^2}2-\frac{\ln(c)}2-c^2\ln(c),&1\leq c
\end{array}\right..
\end{equation}
Plugging in $\phi_2(x)={3/4x^2} - 1/2 x^2 \ln(2|x|)$ and 
$$\tilde{\Omega}_c''(z)=\frac{2 (1 + c z)}{\pi (1 + c^2 + 2 c z) \sqrt{1 - z^2}}$$
splits \eqref{eq:intphi2p''} to two integrals. The first one is
\begin{align}
\label{eq:int-sin}
\frac 3{8\pi}\int_{-1}^1{\frac{(1+c^2+2cz)(1+cz)}{\sqrt{1-z^2}}}dz
&
=\frac 3{8\pi}\int_{-\pi/2}^{\pi/2}{1+c^2+(3c+c^2)\sin\theta+2c^2 \sin^2 \theta}d\theta
\\\nonumber&
=\frac 38 (1+2c^2).
\end{align}
Using this, \eqref{eq:intphi2p''} is equivalent to
\begin{multline}
\label{eq:int-ln2}
\frac 1{4\pi}\int_{-1}^1{\frac{\ln(1+c^2+2cz)(1+c^2+2cz)(1+cz)}{\sqrt{1-z^2}}}dz
\\
=\left\{\begin{array}{ll}
\frac 34 c^2+\frac{c^4}8,&0<c\leq 1\\
\frac 58+\frac {c^2}4+\frac{\ln(c)}2+c^2\ln(c),&1\leq c
\end{array}\right..
\end{multline}
Since both sides of this equation are differentiable in $c$ and the equation is obviously true when $c=0$, it is enough to show that the derivatives of both sides with respect to $c$ agree.

Differentiating the left-hand side of \eqref{eq:int-ln2} with respect to $c$ three times and simplifying reduces \eqref{eq:int-ln2} to
\begin{equation}
\label{eq:intLn}
\frac 1{4\pi}\int_{-1}^1{\frac{\ln(1+c^2+2cz)6z}{\sqrt{1-z^2}}}dz
=
\left\{\begin{array}{ll}
\frac 32 c,&0<c\leq 1\\
\frac 32 \frac 1c,&1\leq c
\end{array}\right..
\end{equation}
Differentiating the left-hand side of \eqref{eq:intLn} with respect to $c$ we obtain

\begin{align*}
\frac 3{2\pi}\int_{-1}^1&{\frac{2(c+z)z}{(1+c^2+2cz)\sqrt{1-z^2}}}dz
\\&\qquad\qquad=
\frac 3{2\pi}\int_{-1}^1{\frac 1{\sqrt{1-z^2}}\(\frac{-1+c^2}{2c^2}+\frac zc+\frac{1-c^4}{4c^3\(\frac{1+c^2}{2c}+z\)}\)}dz
\\&\qquad\qquad=
\left\{\begin{array}{ll}
\frac 32,&0<c\leq 1\\
-\frac 32 \frac 1{c^2},&1\leq c
\end{array}\right..
\end{align*}


This completes the proof of \eqref{eq:intphi2p''} and of the lemma.
$\qed$
\end{proof}

\begin{lemma}
\label{lem:I}
For any $c>0$ we have
\begin{equation}
\label{eq:I}
	I_c(s)=\phi_1(a-s)+\phi_1(b-s)+G_c(s)-H_c(s),
\end{equation}
where $G_c(s)$ is defined by
\begin{equation}
\label{eq:G}
G_c(s):=\frac 1c \phi_1\(\frac{1+2cs}2\)-\frac{1-c^2}{2c}.
\end{equation}
\end{lemma}
\begin{proof}
Integrating $I_c(s)$ by parts we obtain
\begin{equation*}
I_c(s)=-\Omega_c'(t)\phi_1(s-t)\big|_a^{-\frac 1{2c}}-\Omega_c'(t)\phi_1(s-t)\big|_{-\frac 1{2c}}^b+\int_a^b{\phi_1(s-t)\Omega_c''(t)}dt.
\end{equation*}
The first two parts are equal to
\begin{equation}
\label{eq:I1}
\phi_1(a-s)+\phi_1(b-s)+
\left\{\begin{array}{cl}
0,& c<1
\\\phi_1(\frac 1{2c}+s),& c=1
\\2\phi_1(\frac 1{2c}+s),& c>1
\end{array}
\right.
\end{equation}
and the last part in the shifted notation becomes
\begin{equation*}
	\int_{-1}^1{\phi_1(z-w)\tilde{\Omega}_c''(w)}dw,
\end{equation*}
where the integration limits are taken to be $\pm 1$ since $\tilde{\Omega}_c''(w)=0$ outside the interval $[-1,1]$.





Differentiating twice with respect to $z$, decomposing into partial fractions and using \eqref{eq:GRzSmall} and \eqref{eq:GRzSmallc} we obtain


\begin{multline}
\label{eq:ddGH}
\frac {d^2}{dz^2}\(\int_{-1}^1{\phi_1(z-w)\tilde{\Omega}_c''(w)}dw\)
=\frac{-1}{\frac{1+c^2}{2c}+z}\(\delta_{|z|>1}\sign(z)\frac {\frac 1c+z}{\sqrt{z^2-1}}+Sign(1-c)\)
\\=-\sign(1-c)\frac{2c}{1+c^2+2cz}-\delta_{|z|>1}\sign(z)\(\frac 1{\sqrt{z^2-1}}+\frac {\frac 1c -\frac{1+c^2}{2c}}{\(z+\frac{1+c^2}{2c}\)\sqrt{z^2-1}}\).
\end{multline}
Using \eqref{eq:GRzBig} with $\alpha=\frac{1+c^2}{2c}>1$ we can integrate \eqref{eq:ddGH} to obtain
\begin{multline}
\label{eq:dGH}
\int_{-1}^1{\phi_0(z-w)\tilde{\Omega}_c''(w)}dw=\frac {d}{dz}\(\int_{-1}^1{\phi_1(z-w)\tilde{\Omega}_c''(w)}dw\)
\\=-\sign(1-c)\ln(1+c^2+2cz)
-\delta_{|z|>1}\sign(z)\ln|z+\sqrt{z^2-1}|
\\-\delta_{|z|>1}
    \sign(1-c)\arccosh\left|
      \frac{1+\frac{1+c^2}{2c}z}{z+\frac{1+c^2}{2c}}\right|+F_1(c),
\end{multline}
where $F_1(c)$ is a certain function that depends only on $c$. Next, we find $F_1(c)$. Since $F_1(c)$ is independent of $z$, $z$ can be fixed. Seting $z=0$ in \eqref{eq:dGH} and integrating by parts, we obtain
\begin{equation}
\label{eq:F1}
  \int_{-1}^1{\frac 1w\tilde{\Omega}_c'(w)}dw=F_1(c)+
\left\{\begin{array}{ll}
2\ln(2)-\ln(1+c^2),&0\leq c\leq 1\\
\ln(1+c^2),&c\geq 1
\end{array}\right..
\end{equation}

For $c\neq 1$ differentiating both sides with respect to $c$ we obtain
$$
  \frac 2{\pi} \int_{-1}^1{\frac 1w\frac{\sqrt{1-w^2}}{1+c^2+2cw}}dw=-\sign(1-c)\frac{2c}{1+c^2}+F'_1(c).
$$

Since 
$$
  \frac 1{w(1+c^2+2cw)}=\frac {2c}{1+c^2}\(\frac 1{2cw} - \frac 1{1+c^2+2cw}\)
$$
and ${\sqrt{1-w^2}}/w$ is an odd function, it follows that
\begin{equation}
\label{eq:F1p}
\frac 2{\pi} \int_{-1}^1{\frac{\sqrt{1-w^2}}{1+c^2+2cw}}dw=\sign(1-c)-\frac{1+c^2}{2c}F'_1(c).
\end{equation}





Calculating the left-hand side using \cite[2.267, p.97]{GradshteynRyzhik} and \eqref{eq:GRzSmallc} we obtain
\begin{equation*}
\frac 1{c\pi} \int_{-1}^1{\frac{\sqrt{1-w^2}}{\frac{1+c^2}{2c}+w}}dw
=\left\{\begin{array}{ll}
1,&0<c< 1\\
\frac 1{c^2},&c> 1
\end{array}\right..
\end{equation*}


From this calculation and \eqref{eq:F1p} it follows that
$$
F'_1(c)=\left\{\begin{array}{ll}
0,&0<c< 1\\
-\frac 2c,&c> 1
\end{array}\right.,
$$
hence, using that $F_1(c)$ is continuous, we obtain
\begin{equation}
\label{eq:F1const}
F_1(c)=\const+\left\{\begin{array}{ll}
0,&0<c\leq 1\\
-2\ln(c),&c\geq 1
\end{array}\right..
\end{equation}
Setting $c=0$ in \eqref{eq:F1} we obtain that the constant in \eqref{eq:F1const} is $0$.

Integrating \eqref{eq:dGH} with respect to $z$, using integration by parts for the last piece and using \eqref{eq:GRzBig} we obtain
\begin{multline}
\label{eq:GH}
  \int_{-1}^1{\phi_1(z-w)\tilde{\Omega}_c''(w)}dw=
  \sign(1-c)\frac 1c \phi_1\(\frac{1+c^2+2cz}{2}\)-\frac{1-c^2}{2c}
  \\-\delta_{|z|>1}\sign(z)(z\ln|z+\sqrt{z^2-1}|-\sqrt{z^2-1})
  -\delta_{|z|>1}\sign(1-c)z\arccosh\left|
      \frac{1+\frac{1+c^2}{2c}z}{z+\frac{1+c^2}{2c}}\right|
  \\+\delta_{|z|>1}\sign(z)\frac{1-c^2}{2c}\int{\frac{z}{\(z+\frac{1+c^2}{2c}\)\sqrt{z^2-1}}}dz
  -2\delta_{c>1}\ln(c)z+F_2(c).
\end{multline}
Calculating the remaining integral using \eqref{eq:GRzBig} and noting that
$$
  \arccosh|z|=\sign(z)\ln|z+\sqrt{z^2-1}|,
$$
we obtain
\begin{multline}
\label{eq:ph1Omegapp}
\int_{-1}^1{\phi_1(z-w)\tilde{\Omega}_c''(w)}dw
\\
=\tilde{G}_c(z)-
\left\{\begin{array}{cl}
0,& c<1
\\\frac 1c \phi_1(\frac {1+c^2+2cz}2),& c=1
\\\frac 2c \phi_1(\frac {1+c^2+2cz}2),& c>1
\end{array}
\right.
-\tilde{H}_c(z)-2\delta_{c>1}\ln(c)z+F_2(c).
\end{multline}
It remains to find $F_2(c)$. Since $F_2(c)$ is independent of $z$, it suffices to consider the limit ${z\rightarrow -\frac{1+c^2}{2c}}$:
\begin{multline*}
  \int_{-1}^1{\phi_1\(-\frac{1+c^2}{2c}-w\)\tilde{\Omega}_c''(w)}dw\\=\lim_{z\rightarrow -\frac{1+c^2}{2c}}\(\tilde{G}_c(z)-
\left\{\begin{array}{cl}
0,& c<1
\\\frac 1c \phi_1(\frac {1+c^2+2cz}2),& c=1
\\\frac 2c \phi_1(\frac {1+c^2+2cz}2),& c>1
\end{array}
\right.
-\tilde{H}_c(z)-2\delta_{c>1}\ln(c)z+F_2(c)\).
\end{multline*}
Substituting $x-x\ln|2x|$ for $\phi_1(x)$ we obtain
\begin{multline*}
-\frac 1{c\pi}\int_{-1}^1{(1+\ln(c)-\ln(1+c^2+2cw))\frac{1+cw}{\sqrt{1-w^2}}}dw
\\=F_2(c)+\left\{\begin{array}{ll}
\frac{-1+c^2-\ln(c)}{c},&0<c\leq 1\\
\frac{2+c^2}{c}\ln(c),&c\geq 1
\end{array}\right.
\end{multline*}
or, equivalently,
\begin{equation}
\label{eq:F2c}
  \frac 1{\pi}\int_{-1}^1{\ln(1+c^2+2cw)\frac{1+cw}{\sqrt{1-w^2}}}dw=F_2(c)c+\left\{\begin{array}{ll}
c^2,&0\leq c\leq 1\\
1+(3+c^2)\ln(c),&c\geq 1
\end{array}\right..
\end{equation}
Differentiating the left-hand side with respect to $c$ we obtain
\begin{equation}
\label{eq:intLn2}
\frac 1{\pi}\int_{-1}^1{\frac{\ln(1+c^2+2cw)w}{\sqrt{1-w^2}}}dw
  +\frac 1{\pi}\int_{-1}^1{\frac{2(c+w)(1+cw)}{(1+c^2+2cw)\sqrt{1-w^2}}}dw.
\end{equation}





It follows from \eqref{eq:intLn} that the first part of \eqref{eq:intLn2} is equal to $c$ if $0<c\leq 1$ and $1/c$ if $c>1$. The second part is equal to
\begin{equation*}
\frac 1{\pi}\int_{-1}^1{\frac 1{\sqrt{1-w^2}}\(\frac{1+c^2}{2c}+w+\frac{-1+2c^2-c^4}{4c^2}\frac 1{\frac{1+c^2}{2c}+w}\)}dw
=\left\{\begin{array}{ll}
c,&0<c\leq 1\\
\frac{1}{c},&c\geq 1
\end{array}\right..
\end{equation*}


Adding these and integrating with respect to $c$ we obtain the left-hand side of \eqref{eq:F2c} up to a constant. Setting $c=0$ in \eqref{eq:F2c} the constant is easily found, giving
\begin{equation}
\label{eq:F2}
F_2(c)=-\delta_{c>1}\frac{1+c^2}{c}\ln(c).
\end{equation}

Combining \eqref{eq:I1}, \eqref{eq:ph1Omegapp} and \eqref{eq:F2} we obtain
\begin{multline*}
I_c(s)=\phi_1(b-s)+\phi_1(a-s)+G_c(s)-
\left\{\begin{array}{cl}
0,& c<1
\\\frac 1c \phi_1(\frac {1+2cs}2),& c=1
\\\frac 2c \phi_1(\frac {1+2cs}2),& c>1
\end{array}
\right.
-H_c(s)\\+\left\{\begin{array}{cl}
0,& c<1
\\\phi_1(\frac 1{2c}+s),& c=1
\\2\phi_1(\frac 1{2c}+s),& c>1
\end{array}
\right.
+\delta_{c\geq1}\(-2\left(s-\frac c2\right)\ln(c)-\frac{1+c^2}{c}\ln(c)\),
\end{multline*}
which simplifies to \eqref{eq:I}.
$\qed$
\end{proof}

\begin{lemma}
\label{lem:F3}
For any $c>0$ we have
\begin{multline}
\label{eq:F3}
\int_{-1}^1{\phi_2(x-z)\tilde{\Omega}_c''(z)}dz=\sign(1-c)\frac 1{c^2} \phi_2\(\frac{1+c^2+2cx}2\)
-\frac{1-c^2}{2c}x\\-\tilde{J}_c(x)+\frac{-3+4c^2+3c^4}{16c^2}-\delta_{c>1}\(\(x+\frac{1+c^2}{2c}\)^2\ln(c)\),
\end{multline}
where $\tilde{J}_c(z)=0$ if $|z|\leq 1$ and
\begin{multline}
	\tilde{J}_c(z)=\frac 12 \(1-\frac 1{2c^2}+\(z+\frac {c^2-1}{2c}\)^2\)\arccosh|z|
	+\sign(z)\frac{1 - c^2 - 3 c z}{4 c}\\\times\sqrt{z^2-1}
	+\sign(1-c)\frac 12 \(z+\frac {c^2+1}{2c}\)^2 \arccosh\left|
      \frac{1+\frac{1+c^2}{2c}z}{z+\frac{1+c^2}{2c}}\right|,
\end{multline}
if $|z|>1$.
\end{lemma}
\begin{proof}
Differentiating the left-hand side of \eqref{eq:F3} and using \eqref{eq:ph1Omegapp} and \eqref{eq:F2} we obtain
\begin{multline*}
\frac d{dx}\int_{-1}^1{\phi_2(x-z)\tilde{\Omega}_c''(z)}dz
=\int_{-1}^1{\phi_1(x-z)\tilde{\Omega}_c''(z)}dz
\\=\tilde{G}_c(x)-
\left\{\begin{array}{cl}
0,&c<1
\\\frac 1c \phi_1\(\frac{1+c^2+2cx}2\),&c=1
\\2\frac 1c \phi_1\(\frac{1+c^2+2cx}2\),&c>1
\end{array}\right.
-\tilde{H}_c(x)-\delta_{c>1}\(2\ln(c)x+\frac{1+c^2}{c}\ln(c)\),
\end{multline*}
whence
\begin{multline*}
  \int_{-1}^1{\phi_2(x-z)\tilde{\Omega}_c''(z)}dz=\sign(1-c)\frac 1{c^2} \phi_2\(\frac{1+c^2+2cx}2\)
  -\frac{1-c^2}{2c}x\\-\int{\tilde{H}_c(x)}dx-\delta_{c>1}\(\ln(c)x^2+\frac{1+c^2}{c}\ln(c)x\).
\end{multline*}
Since $\frac d{dx}\tilde{J}_c(x)=\tilde{H}_c(x)$ for $|x|>1$, we obtain
\begin{multline*}
  \int_{-1}^1{\phi_2(x-z)\tilde{\Omega}_c''(z)}dz=\sign(1-c)\frac 1{c^2} \phi_2\(\frac{1+c^2+2cx}2\)
  -\frac{1-c^2}{2c}x-\tilde{J}_c(x)\\-\delta_{c>1}\(x^2\ln(c)+\frac{1+c^2}{c}x\ln(c)\)+F_4(c).
\end{multline*}
To find $F_4(c)$, find the limit of both sides when $x\rightarrow -\frac{1+c^2}{2c}$. The left-hand side becomes
\begin{multline*}
  \int_{-1}^1{\phi_2\(\frac{1+c^2+2cz}{2c}\)\tilde{\Omega}_c''(z)}dz
  =\frac 1{8c^2\pi}\int_{-1}^1{(3+2\ln(c))\frac{(1+cz)(1+c^2+2cz)}{\sqrt{1-z^2}}}dz
  \\-\frac 1{8c^2\pi}\int_{-1}^1{2\ln(1+c^2+2cz)\frac{(1+cz)(1+c^2+2cz)}{\sqrt{1-z^2}}}dz.
\end{multline*}
The first integral is given in \eqref{eq:int-sin} and the second in \eqref{eq:int-ln2}, thus the left-hand side is





\begin{equation*}
\left\{\begin{array}{ll}
\frac{3-c^4+(2+4c^2)\ln(c)}{8c^2},&0<c\leq1\\
-\frac{1-2c^2+\ln(c)+2c^2\ln(c)}{4c^2},&1\leq c
\end{array}\right..
\end{equation*}


The limit of the right-hand side is
$$
F_4(c)+
\left\{\begin{array}{ll}
\frac{9-4c^2-5c^4+(4+8c^2)\ln(c)}{16c^2},&0<c\leq1\\
\frac{-1+4c^2-3c^4+4c^4\ln(c)}{16c^2},&1\leq c
\end{array}\right.,
$$
whence
$$
  F_4(c)=\frac{-3+4c^2+3c^4}{16c^2}-\delta_{c>1}\(\frac{1+c^2}{2c}\)^2\ln(c).
$$
$\qed$
\end{proof}

\begin{lemma}
\label{lem:intIOmega'}
For any $c>0$ we have
\begin{multline}
\label{eq:intIOmega'}
\int_a^b{I_c(s)\Omega_c'(s)}ds=1-\frac{c^2}4-2\phi_2(b-a)+2\int_a^b{G_c(s)\Omega_c'(s)}ds\\-2\int_a^b{H_c(s)\Omega_c'(s)}ds+
\delta_{c>1}\(1-\frac 5{4c^2}+\frac{c^2}4-\(2+\frac 1{c^2}\)\ln(c)\).
\end{multline}
\end{lemma}
\begin{proof}
Differentiating both sides of \eqref{eq:intIOmega'} with respect to $b$ and noting that $\phi_1$ is odd since it is the integral of an even function, and that $\Omega_c'(b)=1$ since $b>1+c/2$, we obtain
\begin{align*}
	\frac{d}{db}(\text{left-hand side}) &=I_c(b)\Omega_c'(b)+\int_a^b{\phi_0(b-s)\Omega_c'(s)}ds=I_c(b)+I_c(b)\\&=2(\phi_1(0)+\phi_1(a-b)+
	G_c(b)-H_c(b))\\&=-2\phi_1(b-a)+2G_c(b)\Omega_c'(b)-2H_c(b)\Omega_c'(b)
	\\&=\frac{d}{db}(\text{right-hand side}).
\end{align*}
This implies that 
\begin{equation}
\label{eq:intIOmega'F}
\int_a^b{I_c(s)\Omega_c'(s)}ds=1-\frac{c^2}4-2\phi_2(b-a)+2\int_a^b{G_c(s)\Omega_c'(s)}ds-2\int_a^b{H_c(s)\Omega_c'(s)}ds+F_3(c)
\end{equation}
for some function $F_3(c)$ which will be found next. Based on the above argument $F_3(c)$ might depend on $a$, but by symmetry between $a$ and $b$ it does not. Since $F_3(c)$ is independent of $a$ and $b$, $a$ and $b$ can be fixed. Set $a=-{1}/(2c)$ and $b=1+ c/2$, switch to the shifted notation and collect all the integrals in \eqref{eq:intIOmega'F} on the left side:
\begin{multline*}
\int_{-\frac{1+c^2}{2c}}^1 \(\phi_1(1-z)+\phi_1\(-\frac{1+c^2}{2c}-z\)-\tilde{G}_c(z)+\tilde{H}_c(z)\)\tilde{\Omega}_c'(z)dz
\\=1-\frac{c^2}4-2\phi_2\(1+\frac{1+c^2}{2c}\)+F_3(c).
\end{multline*}
Substituting the formula for $\tilde{G}_c(z)$ and integrating by parts we obtain
\begin{align*}
\int_{-\frac{1+c^2}{2c}}^1{\frac{1-c^2}{2c}\tilde{\Omega}'_c(z)}dz
&-\int_{-\frac{1+c^2}{2c}}^1\(-\phi_2(1-z)-\phi_2\(-\frac{1+c^2}{2c}-z\)
\right.\\&\left.
\qquad\qquad\qquad\quad-\frac 1{c^2}\phi_2\(\frac{1+c^2+2cz}{2}\)+\tilde{J}_c(z)\)\tilde{\Omega}''_c(z)dz
\\&-\sign(c-1)\tilde{J}_c\(-\frac{1+c^2}{2c}\)+\sign(c-1)\phi_2\(\frac{1+c^2}{2c}+1\)
\\&-\phi_2\(\frac{1+c^2}{2c}+1\)-\frac 1{c^2}\phi_2\(\frac{1+c^2+2c}{2}\)
\\&=1-\frac{c^2}4-2\phi_2\(1+\frac{1+c^2}{2c}\)+F_3(c),
\end{align*}
which, after simplifications using $\tilde{\Omega}''_c(z)=0$ for $|z|>1$, $\tilde{J}_c(z)=0$ for $|z|\leq 1$,
$$
\int_{-\frac{1+c^2}{2c}}^1{\frac{1-c^2}{2c}\tilde{\Omega}'_c(z)}dz
=\frac{1-c^2}{2c}\(1+\frac c2-\frac 1{2c}\)
$$
and
$$
\sign(c-1)\tilde{J}_c\(-\frac{1+c^2}{2c}\)=\frac 12 \(1+\frac 1{2c^2}\)\ln(c)+\frac{5+c^2}{8c}\frac{1-c^2}{2c},
$$
becomes
\begin{multline*}
\int_{-1}^1\(\phi_2(1-z)+\phi_2\(\frac{1+c^2}{2c}+z\)+\frac 1{c^2}\phi_2\(\frac{1+c^2+2cz}{2}\)\)\tilde{\Omega}''_c(z)dz
\\=\frac 1{c^2}\phi_2\(\frac{1+c^2+2c}{2}\)+\frac 12 \(1 + \frac 1{2 c^2}\)\ln(c)
+\frac{9 - 8 c + 4 c^2 + 8 c^3 - c^4}{16 c^2}
\\-\left\{\begin{array}{cl}
0&c<1
\\\phi_2\(\frac{1+c^2+2c}{2c}\)&c=1
\\2\phi_2\(\frac{1+c^2+2c}{2c}\)&c>1
\end{array}\right.
+F_3(c).
\end{multline*}

The remaining integrals are given by \eqref{eq:intphi2p''} and by Lemma \ref{lem:F3} with $x=-\frac{1+c^2}{2c}$ and $x=1$. Calculating those integrals and simplifying we obtain
$$F_3(c)=
\left\{\begin{array}{ll}
0,&0<c\leq 1\\
1-\frac 5{4c^2}+\frac{c^2}4-\(2+\frac 1{c^2}\)\ln(c),&1\leq c
\end{array}\right..$$
$\qed$
\end{proof}

\section{Proof of the Main Theorems}
\label{sec:MainProofs}
\textit{The upper bound.} The number of Young diagrams with $n$ cells and at most $N$ rows is less than the number $p(n)$ of Young diagrams with $n$ cells. By the Hardy-Ramanujan formula \cite{HardyRamanujan},\cite[p. 116]{Hardy} $p(n)$ is asymptotically given by $p(n)\approx\frac{1}{4n\sqrt{3}}e^{\frac{2\pi}{\sqrt{6}}\sqrt{n}}$. Hence,
$$
\mathbb{P}_N^n\left\{\lambda: \mathbb{P}_N^n(\lambda)<e^{-\frac{2\pi}{\sqrt{6}}\sqrt{n}}\right\}\leq p(n)e^{-\frac{2\pi}{\sqrt{6}}\sqrt{n}}\xrightarrow{n\rightarrow \infty}0.
$$
This implies that
$$\lim_{n\rightarrow \infty}\mathbb{P}_N^n\left\{\lambda: \mathbb{P}_N^n(\lambda)>e^{-\frac{2\pi}{\sqrt{6}}\sqrt{n}}\right\}=1$$
or equivalently that
$$\lim_{n\rightarrow \infty}\mathbb{P}_N^n\left\{\lambda: -\frac 1{\sqrt{n}} \ln\frac{\dim E_\lambda}{N^n}<\frac{2\pi}{\sqrt{6}}\right\}=1.$$
From this it is immediate that 
$$-\frac{1}{\sqrt{n}} \ln \frac{\max\{\dim E_\lambda\}}{N^n}<\frac{2\pi}{\sqrt{6}}$$
for large enough $n$.

\textit{The lower bound.} 
By Propositions \ref{prop:measure} and \ref{prop:integral}, for any $\lambda\in\mathbb{Y}_N^n$ we have
\begin{equation}
\label{eq:lnPNnFinal}
-\frac{\ln \mathbb{P}_N^n(\lambda)}{\sqrt{n}} = \sqrt{n}\left(\frac 12 \|f\|_{\frac 12}^2+2\int_{|s-\frac c2|>1}H_c'(s)f(s)ds\right)+\hat{\theta}(\lambda)-\hat{\rho}(\lambda)-\varepsilon_n,
\end{equation}
where $f(s)=L_\lambda(s)-\Omega_c(s)$. Let $0<i\leq N$ be such that $\lambda_i>0$, where $\lambda_i$ is the length of the $i$-th row of the Young diagram $\lambda$. Let $h_\lambda$ denote the height of $\lambda$, i.e. the number of nonzero rows. Note that $h\leq N$.
\begin{figure}[ht]
\centering
\includegraphics[width=7cm]{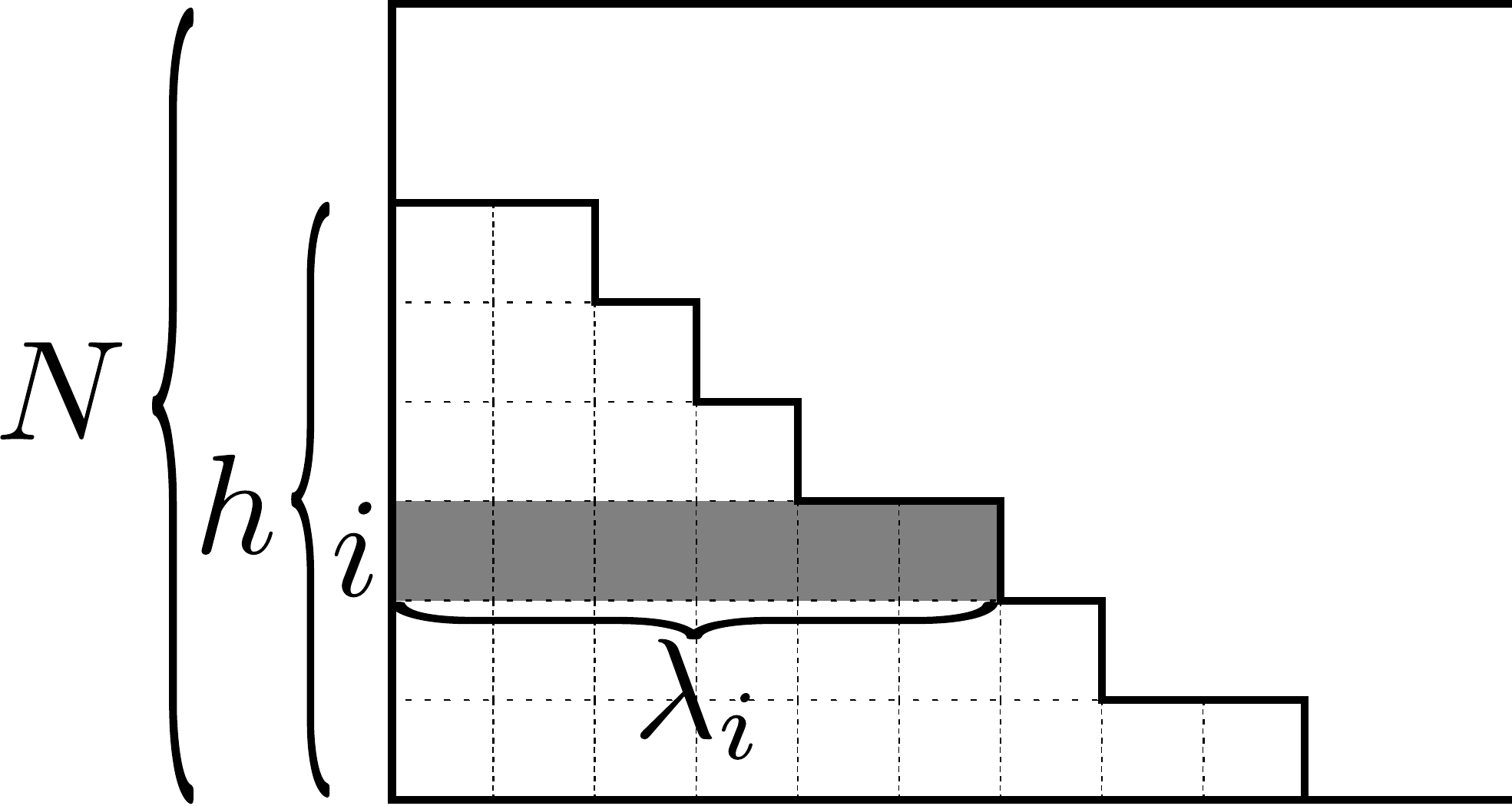}
\caption{\label{fig:HookContent}}
\end{figure}
It is easy to see that $N+c_{i,\lambda_i}\geq h_\lambda+c_{i,\lambda_i}=h_\lambda-i+\lambda_i = h_{i,1}$, i.e. that the shifted content of the last cell in a row of a Young diagram is at least as large as the hook length of the first cell of the row (see Figure \ref{fig:HookContent}). For a fixed $i$ the value of the shifted content $N+c_{i,j}$ decreases by $1$ from right to left, while the length of the hook $h_{i,j}$ decreases by at least one from left to right. Since $m(x)$ is a decreasing function, this implies that $\hat{\theta}(L_\lambda)\geq\hat{\rho}(L_\lambda)$.

It was proven in Corollary \ref{cor:UniqueMinimizer} that $\int_{|s-\frac c2|>1}H_c'(s)f(s)ds\geq 0$. We will not give a lower bound for $\|f\|_{\frac 12}^2$. For our purposes a very rough estimate suffices: we only consider the contribution of diagonal slices in the double integral $\|f\|_{\frac 12}^2$ and use that $L_\lambda$ is piecewise linear with slopes $\pm 1$.

Define $s_i:={i}/({2\sqrt{n}})$. Define
$s^*_i:=min_{s_i\leq s\leq s_{i+1}}(f'(s))^2$.
It follows that
\begin{align*}
\frac{\sqrt{n}}2\|f\|^2
&
\geq\frac{\sqrt{n}}2\sum_i\iint_{s_i\leq s,t\leq s_{i+1}}\(\frac{f(s)-f(t)}{s-t}\)^2 ds dt\geq \frac{\sqrt{n}}2\sum_i (f'(s^*_i))^2 (\Delta s_i)^2
\\&
=\frac{1}{4}\sum_i (f'(s^*_i))^2 \Delta s_i.
\end{align*}
Replacing the Riemann sum by the corresponding integral and using that the sum of all the other terms in \eqref{eq:lnPNnFinal} is nonnegative, we obtain that for any $\varepsilon>0$, there exists $n_0\in\mathbb{N}$ such that for all $n>n_0$ 
\begin{equation*}
-\frac{\ln \mathbb{P}_N^n(\lambda)}{\sqrt{n}}\geq \frac 14 \int_{-\infty}^{\infty}(L_\lambda'(z)-\tilde{\Omega}'_c(z))^2 dz-\varepsilon.
\end{equation*}
Since $L_\lambda$ is linear in the intervals $(s_i,s_{i+1})$ and $L_\lambda'(s)=\pm 1$ for $s\in(s_i,s_{i+1})$, we obtain
\begin{equation*}
-\frac{\ln \mathbb{P}_N^n(\lambda)}{\sqrt{n}}\geq \frac 14 \int_{-\infty}^{\infty}(\sign(\tilde{\Omega}'_c(z))-\tilde{\Omega}'_c(z))^2 dz-\varepsilon 
= \frac 14 \int_{-1}^1(\sign(z)-\tilde{\Omega}'_c(z))^2 dz-\varepsilon,
\end{equation*}
which implies the lower bounds in Theorems \ref{thm:max} and \ref{thm:meas} with 
$$\alpha_c=\frac 14 \int_{-1}^1(\sign(z)-\tilde{\Omega}'_c(z))^2 dz.$$

\bibliographystyle{alpha}
\bibliography{mybib}

\begin{thebibliography}{Wey39}

\bibitem[Bia01]{Biane2001}
P.~Biane.
\newblock Approximate factorization and concentration for characters of
  symmetric groups.
\newblock {\em Int. Math. Res. Notices.}, 2001(4):179--192, 2001.

\bibitem[Buf10]{Bu}
A.~I. Bufetov.
\newblock On the {Vershik-Kerov} conjecture concerning the
  {Shannon-Macmillan-Breiman} theorem for the plancherel family of measures on
  the space of young diagrams.
\newblock 2010.
\newblock arXiv:1001.4275v1 [math.RT].

\bibitem[FH91]{FultonHarris}
W.~Fulton and J.~Harris.
\newblock {\em Representation theory. A first course.}, volume 129 of {\em
  Graduate Texts in Mathematics}.
\newblock Springer-Verlag, New York, 1991.

\bibitem[GR07]{GradshteynRyzhik}
I.~S. Gradshteyn and I.~M. Ryzhik.
\newblock {\em Table of integrals, series, and products}.
\newblock Academic Press/Elsevier, Amsterdam, 7th edition, 2007.

\bibitem[Har99]{Hardy}
G.~H. Hardy.
\newblock {\em Ramanujan: Twelve Lectures on Subjects Suggested by His Life and
  Work}.
\newblock Chelsea, New York, 3rd edition, 1999.

\bibitem[HR18]{HardyRamanujan}
G.~H. Hardy and S.~Ramanujan.
\newblock Asymptotic formulae in combinatory analysis.
\newblock {\em Proc. London Math. Soc.}, 17:75--115, 1918.

\bibitem[LS77]{LoganShepp}
B.~F. Logan and L.~A. Shepp.
\newblock A variational problem for random {Young} tableaux.
\newblock {\em Acta Mathematica}, 26(2):206--222, 1977.

\bibitem[Mac95]{MacDonald}
I.~G. Macdonald.
\newblock {\em Symmetric functions and {Hall} polynomials}.
\newblock Oxford University Press, New York, 2nd edition, 1995.

\bibitem[Ols09]{OlshNotes}
G.~Olshanski.
\newblock Asymptotic representation theory: Lectures at independent university
  of moscow ii.
\newblock {\em Lecture Notes}, 2009.
\newblock http://www.iitp.ru/en/userpages/88/.

\bibitem[VK77]{VK77}
A.~M. Vershik and S.~V. Kerov.
\newblock Asymptotics of the plancherel measure of the symmetric group.
\newblock {\em Soviet Math. Dokl.}, 18:527--531, 1977.

\bibitem[VK85]{VK85}
A.~M. Vershik and S.~V. Kerov.
\newblock Asymptotic behavior of the maximum and generic dimensions of
  irreducible representations of the symmetric group.
\newblock {\em Funktsional. Anal. i Prilozhen.}, 19(1):25--36, 1985.

\bibitem[Wey39]{W}
H.~Weyl.
\newblock {\em The Classical Groups: Their Invariants and Representations}.
\newblock Princeton University Press, Princeton, N.J., 1939.

\end{thebibliography}

\end{document}